\documentclass[a4paper]{article}

\usepackage[english]{babel}
\usepackage[utf8x]{inputenc}
\usepackage[T1]{fontenc}
\usepackage{lipsum}
\usepackage{blindtext}
\usepackage{graphicx,amssymb,amsmath,textcomp}
\usepackage[a4paper,top=3cm,bottom=2cm,left=3cm,right=3cm,marginparwidth=1.75cm]{geometry}
\usepackage{tikz}
\usetikzlibrary{matrix,arrows,positioning,calc}
\usepackage{verbatim}

\usepackage{amsmath,amsthm}
\usepackage{faktor}
\usepackage{xfrac} 
\usepackage{graphicx}
\usepackage[colorinlistoftodos]{todonotes}
\usepackage[colorlinks=true, allcolors=blue]{hyperref}

\usepackage{mathtools}

\numberwithin{equation}{section} %or {subsection}

\newtheorem{corollary}{Corollary}[subsection] % or [theorem] or [section]
\newtheorem{lemma}{Lemma}[subsection]

\theoremstyle{proposition}
\newtheorem{proposition}{Proposition}[subsection]

\theoremstyle{definition}
\newtheorem{definition}{Definition}[subsection]
\newtheorem{remark}{Remark}[subsection]
\newtheorem{example}{Example}[subsection]

\title{\textbf{Notes on Moduli theory, Stacks and 2-Yoneda's Lemma}}
\author{Kadri \.{I}lker Berktav %\footnote{\textit{E-mail:} berktav@metu.edu.tr} \\ %\textit{Department of Mathematics, Middle East Technical University}, \\ \textit{06800 Ankara, Turkey}
} 
\date{\vspace{-5ex}}

\begin{document}
\maketitle

\begin{abstract}
This note is a survey on the basic aspects of moduli theory along with some examples. In that respect, one of the purposes of this current document is to understand how the introduction of stacks circumvents the non-representability problem of the corresponding moduli functor $\mathcal{F}$  by using the 2-category of stacks. To this end, we shall briefly revisit the basics of 2-category theory and %we would like to 
present a 2-categorical version of  Yoneda's lemma for the "refined" moduli functor $\mathcal{F}$. Most of the material below are standard and can be found elsewhere in the literature. For an accessible introduction to moduli theory and stacks, we refer to \cite{BenZ, Neumann}. For an extensive treatment to the case of moduli of curves, see \cite{Hain1, Hain2, Harris and Morrison}. Basics of 2-category theory and further discussions can be found in \cite{Stacks}.

\end{abstract}
\tableofcontents

    \section{Functor of points, representable functors, and Yoneda's Lemma}
%These concepts essentially play important roles in the theory of moduli spaces in the sense that 
Main aspects of a moduli problem of interest can be encoded by a certain functor, namely \textit{a moduli functor} of the form
\begin{equation}
\mathcal{F}: \mathcal{C}^{op} \longrightarrow \ Sets
\end{equation} where $ \mathcal{C}^{op} $ is the \textit{opposite} category of the category $\mathcal{C}$, and $Sets$ denotes the category of sets. Equivalently, it is just a \textit{contravariant functor} from the category $\mathcal{C}$ to $Sets$. The existence of \textit{a fine moduli space}  corresponds to the representability of this moduli functor. More details are to be discussed below.

\begin{definition}
	Let $\mathcal{C}$ be a category. For any object $U$ in $\mathcal{C}$ we define a functor 
	\begin{equation}
	h_U: \mathcal{C}^{op} \longrightarrow \ Sets
	\end{equation} as follows: 
	\begin{enumerate}
		\item For each object $X\in Ob(\mathcal{C})$, \ $X \longmapsto h_U(X):=Mor_{\mathcal{C}}(X,U)$ 
		\item For each morphism $X\xrightarrow{f}Y$,  \begin{equation*}
		\Big(X\xrightarrow{f}Y \Big) \ \longmapsto \Big(Mor_{\mathcal{C}}(Y,U) \xrightarrow{f^*} Mor_{\mathcal{C}}(X,U), \ g\mapsto g \circ f  \Big).
		\end{equation*} 
	\end{enumerate} This functor $h_U$ is called "the Yoneda functor" or "the functor of points".
\end{definition}
As we shall see below, this functor can be used to recover any object $U$ of a category $\mathcal{C}$ by understanding  morphisms into it via  Yoneda's Lemma.  Let $ Fun(\mathcal{C}^{op}, Sets) $ denote the category of functors from $\mathcal{C}^{op}$ to $Sets$ with objects being functors $ \mathcal{F}: \mathcal{C}^{op} \longrightarrow \ Sets $ and morphisms being natural transformations between functors, then using the definition of $h_U$, one can introduce the following functor as well:
\begin{equation} \label{defn_Yonede functor h}
h: \mathcal{C}\longrightarrow \ Fun(\mathcal{C}^{op}, Sets)
\end{equation} where
\begin{enumerate}
	\item For each object $U\in Ob(\mathcal{C})$, \ $U \longmapsto h_U=Mor_{\mathcal{C}}(\cdot,U)$ 
	\item To each morphism $U\xrightarrow{f}V$,  $h$ assigns a natural transformation 
	\begin{equation}
	{\small \begin{tikzpicture}
		\matrix[matrix of nodes,column sep=1.7cm] (cd)
		{
			$ \mathcal{C}^{op} $ & $ Sets. $ \\
		};
		\draw[->] (cd-1-1) to[bend left=50] node[label=above:$ {h_U} $] (U) {} (cd-1-2);
		\draw[->] (cd-1-1) to[bend right=50,name=D] node[label=below:${ h_V}$] (V) {} (cd-1-2);
		\draw[double,double equal sign distance,-implies,shorten >=10pt,shorten <=10pt] 
		(U) -- node[label=left:${\tiny h_f}$] {} (V);
		\end{tikzpicture}}
	\end{equation} Here $h_f$ is defined as follows:
	\begin{enumerate}
		\item For each object $X$ in $\mathcal{C}$, we set \begin{equation}
		h_f(X): h_U(X) \rightarrow h_V(X), \ \ g \longmapsto f \circ g .
		\end{equation}
		\item Given a morphism $X\xrightarrow{\eta} Y$ in $\mathcal{C}$, from the associativity property of the composition map, %i.e. $f\circ(g\circ \eta)=(f \circ g)\circ \eta$, 
		the diagram 
		\begin{equation}
		\begin{tikzpicture}
		\matrix (m) [matrix of math nodes,row sep=3em,column sep=4em,minimum width=2em] {
			h_U(Y)  & h_V(Y) \\
			h_U(X) & h_V(X)\\};
		\path[-stealth]
		(m-1-1) edge  node [left] {{\small $ \circ \eta $}} (m-2-1)
		edge  node [above] {{\small $  f \circ $}} (m-1-2)
		(m-2-1.east|-m-2-2) edge  node [below] {} node [below] {{\small $ f \circ $}} (m-2-2)
		(m-1-2) edge  node [right] {{\small $ \circ \eta $}} (m-2-2);
		%edge [dashed,-] (m-2-1);
		\end{tikzpicture}
		\end{equation} commutes.
	\end{enumerate} 
\end{enumerate} 
\begin{definition} \label{defn_fullyfaithfulness and essentially surje}
	\cite{Vakil} Let $\mathcal{C}, \mathcal{D}$ be two categories.
	\begin{enumerate}
		\item A functor $\mathcal{F}:\mathcal{C}\rightarrow \mathcal{D}$ is called \textit{fully faithful} if for any objects $A,B \in \mathcal{C}$, the map \begin{equation}
		Hom_{\mathcal{C}}(A,B)\longrightarrow Hom_{\mathcal{D}}(\mathcal{F}(A),\mathcal{F}(B))
		\end{equation} is a bijection of sets.
		\item A functor $\mathcal{F}:\mathcal{C}\rightarrow \mathcal{D}$ is called \textit{essentially surjective} if for any objects $D \in \mathcal{D}$, there exists an object $A$ in $ \mathcal{C} $ such that one has an isomorphism of objects \begin{equation}
		\mathcal{F}(A)\xrightarrow{\sim}D.
		\end{equation}
	\end{enumerate}
\end{definition}
\begin{lemma} \label{Yoneda's lemma}
	(Yoneda's Lemma) The functor $h$ above is  fully faithful.
\end{lemma}

\begin{remark} \label{rmks on Yoneda's lemma}
	\begin{enumerate}
		\item []
		\item Yoneda's lemma implies that the functor $h$ serves as an embedding (sometimes it is also called \textit{Yoneda's embedding}), and hence $h_U$ determines $U$ up to a unique isomorphism. Therefore, one can recover any object $U$ in $\mathcal{C}$ by just knowing all possible morphisms into $U$. In the case of the category $Sch_{\mathbb{C}}$ of $\mathbb{C}$-schemes, for instance, it is enough to study the restriction of this functor to the full subcategory $Aff_{\mathbb{C}}$  of affine $\mathbb{C}$-schemes, in order to recover the scheme $U$.
		\item Thanks to the Yoneda's embedding, one can also realize some algebro-geometric objects (like schemes, stacks, derived "spaces", etc...) as \textit{a certain functor} in addition to standard ringed-space formulation. We have the following enlightening diagram \cite{Vezz2} encoding such a functorial interpretation: 
		\begin{center}
			\begin{tikzpicture}
			\matrix (m) [matrix of math nodes,row sep=2em,column sep=4em,minimum width=2 em] {
				Sch^{op}\cong CAlg_k   & Sets  \\
				&  Grpds \\
				cdga_k^{<0} &  Ssets \\};
			\path[-stealth]
			(m-1-1) edge  node [left] {{\small simplicial enrichments}} (m-3-1)
			edge  node [above] {{\small Schemes}} (m-1-2)
			(m-1-1) edge  node [below] {} node [below] {{\small stacks}} (m-2-2)
			(m-1-1) edge  node [below] {} node [below] {{\small n-stacks}} (m-3-2)
			(m-3-1) edge  node  [below] {{\small derived stacks}} (m-3-2)
			
			(m-1-2) edge  node [right] {{\small 2-categorical refinements}} (m-2-2)
			(m-2-2) edge  node [right] {{\small higher categorical refinements}} (m-3-2);
			%edge [dashed,-] (m-2-1);
			\end{tikzpicture}
		\end{center} One way of interpreting this diagram is as follows: In the case of schemes (stacks resp.), for instance, such a functorial description implies that  points of a scheme (a stack resp.) $X$ form a \textit{set} (a \textit{groupoid} resp.). These kind of interpretations, in fact, suggest the name "functor of points".
		\item The bad news is that not all functors $ \mathcal{F}: \mathcal{C}^{op} \longrightarrow \ Sets $ are of the form $h_U$ for some $U$ in a general set-up. In other words, $h$ is \textit{not} essentially surjective in general. This in fact leads to the following definition: 
	\end{enumerate}
\end{remark}

\begin{definition}
	A functor $ \mathcal{F}: \mathcal{C}^{op} \longrightarrow \ Sets $ is called \textit{representable} if there exists $\mathcal{M}\in Ob(\mathcal{C})$ such that we have a natural isomorphism $\mathcal{F}\Leftrightarrow h_\mathcal{M}.$ That is, \begin{equation}
	\mathcal{F} = Mor_{Sch_{\mathbb{C}}} (\cdot, \mathcal{M}) \ \ \text{for  some} \ \mathcal{M}\in Ob(\mathcal{C}).
	\end{equation}If this is the case, then we say that \textit{$\mathcal{F}$ is represented by $\mathcal{M}$}.  In the case of moduli theory, $\mathcal{M}$ is then called a \textbf{\textit{fine moduli space}}. In the next section, we shall investigate the properties of $\mathcal{M}$. %elaborate the fact that this $\mathcal{M}$, in fact, satisfies the properties for a space to be declared as a fine moduli space of a given moduli/classification problem of interest.
\end{definition}
\section{Moduli theory in functorial perspective}
%\subsection{Preliminaries}
\textit{A moduli problem} is a problem of constructing a classifying space (or \textit{a moduli space} $\mathcal{M}$) for certain geometric objects (such as manifolds, algebraic varieties, vector bundles etc...) up to their intrinsic symmetries. In other words, a moduli space serves as a solution space of a given moduli problem of interest. In general, the set of isomorphism classes of objects that we would like to classify may not be able to provide a sufficient information to encode  geometric properties of the moduli space itself. Therefore, we expect a moduli space to behave well enough to capture the underlying geometry. Thus, this expectation leads to the following wish-list for $\mathcal{M}$ to be declared as a "fine" moduli space:
\begin{enumerate}
	\item  $\mathcal{M}$ is supposed to serve as a \textit{parameter space} in a sense that there must be a one-to-one correspondence between the points of  $\mathcal{M}$ and the \textit{set} of isomorphism classes of objects to be classified:
	\begin{equation}
	\{points \ of \ \mathcal{M}\} \leftrightarrow \{isomorphism \ classes \ of \ objects \ in \ \mathcal{C} \}
	\end{equation}
	\item One ensures the existence of \textit{universal classifying object}, say $\mathcal{T}$, through which all other objects parametrized by $\mathcal{M}$ can also be reconstructed. This, in fact, makes the moduli space $\mathcal{M}$ even more sensitive to the behavior of "families" of objects on any base object $B$. It is manifested by a certain representative morphism $B\rightarrow \mathcal{M}$. That is, for any  family  \begin{equation}
	\pi: X \longrightarrow B
	\end{equation}   parametrized by some base scheme $ B$ where $$ X:= \{X_b \in Ob(\mathcal{C}) : \pi^{-1}(b)=X_b, \ b\in B \}, $$ there exits a unique morphism $f:B \rightarrow \mathcal{M}$ such that one has the following fibered product diagram: 
	\begin{equation}
	\begin{tikzpicture}
	\matrix (m) [matrix of math nodes,row sep=3em,column sep=4em,minimum width=2em] {
		X  & \mathcal{T} \\
		B & \mathcal{M}\\};
	\path[-stealth]
	(m-1-1) edge  node [left] {{\small $ \pi $}} (m-2-1)
	(m-1-1.east|-m-1-2) edge  node [above] {{\small $  $}} (m-1-2)
	(m-2-1.east|-m-2-2) edge  node [below] {} node [below] {{\small $ f $}} (m-2-2)
	(m-1-2) edge  node [right] {{\small $  $}} (m-2-2);
	%edge [dashed,-] (m-2-1);
	\end{tikzpicture}
	\end{equation}where $X=B \times_{\mathcal{M}} \mathcal{T}$. That is, the family $X$ can be uniquely obtained by pulling back the universal object $\mathcal{T}$ along the morphism $f.$ 
\end{enumerate}
For an accessible overview, see \cite{BenZ}. Relatively complete treatments can be found in \cite{Hoskins,Neumann}.\begin{remark}
	More formally, \textit{a family over a base $B$} is a scheme $X$ together with a morphism $\pi: X \rightarrow B$ of schemes where for each (closed point) $b\in B$ the fiber  $X_b$ is defined as fibered product \begin{equation}
	\begin{tikzpicture}
	\matrix (m) [matrix of math nodes,row sep=3em,column sep=4em,minimum width=2em] {
		X_b= \{b \} \times_{B} X  & X \\
		\{b \} & B\\};
	\path[-stealth]
	(m-1-1) edge  node [left] {{\small $  $}} (m-2-1)
	(m-1-1.east|-m-1-2) edge  node [above] {{\small $  $}} (m-1-2)
	(m-2-1.east|-m-2-2) edge  node [below] {} node [below] {{\small $ \imath $}} (m-2-2)
	(m-1-2) edge  node [right] {{\small $ \pi $}} (m-2-2);
	%edge [dashed,-] (m-2-1);
	\end{tikzpicture} 
	\end{equation} where $\imath: \{b \} \hookrightarrow B $ is the usual inclusion map.
\end{remark}

 In the language of category theory, on the other hand, a moduli problem can be formalized as a certain functor \begin{equation}
\mathcal{F}:  \mathcal{C}^{op} \longrightarrow Sets,
\end{equation}
which is called \textit{\textbf{a moduli functor}} where $ \mathcal{C}^{op} $ denotes the \textit{opposite} category of the category $\mathcal{C}$, and $ Sets $ is the category of sets. In other words, it is just a \textit{contravariant functor} from the category $\mathcal{C}$ to $Sets$. In order to make the argument more transparent, we take $\mathcal{C}$ to be the category $Sch_{\mathbb{C}}$ of $\mathbb{C}$-schemes unless otherwise stated. Note that 
for each  $\mathbb{C}$-scheme $ U \in Sch $, \ $\mathcal{F}(U)$ is the \textit{set} of isomorphism classes (of families) parametrized by $U$. For each morphism $f:U \rightarrow V$ of schemes, we have a morphism $\mathcal{F}(f): \mathcal{F}(V) \rightarrow \mathcal{F}(U)$ of sets. 

\begin{example}
	Given a scheme $U$, one can define $\mathcal{F}(U):=S(U)/\sim$ where $S(U)$ is \textit{the set of families over the base scheme $U$} \begin{equation*}
	S(U):= \bigg\{ X \rightarrow U: X \text{ is a scheme over } U, \text{each fiber $ X_u $ is $ C_g \ \forall u\in U$} \bigg\}
	\end{equation*} where $ C_g $ is a smooth, projective, algebraic curve of genus g. We say that two families $\pi: X \rightarrow U$ and $\pi': Y \rightarrow U$ over $U$ are \textit{equivalent} if and only if there exists an isomorphism $f: X \xrightarrow{\sim} Y$ of schemes such that the following diagram commutes:
	\begin{equation}
	\begin{tikzpicture}
	\matrix (m) [matrix of math nodes,row sep=2em,column sep=3em,minimum width=1em] {
		X  & & Y \\
		& U &\\};
	\path[-stealth]
	(m-1-1) edge  node [left] {{\small $ \pi $}} (m-2-2)
	(m-1-1.east|-m-1-3) edge  node [above] {{\small $ f $}} (m-1-3)
	%(m-2-1.east|-m-2-2) edge  node [below] {} node [below] {{\small $ f $}} (m-2-2)
	(m-1-3) edge  node [right] {{\small $ \pi' $}} (m-2-2);
	%edge [dashed,-] (m-2-1);
	\end{tikzpicture}
	\end{equation}

\noindent On morphisms $\phi: U \rightarrow V$, on the other hand, we have
\begin{equation}
\mathcal{F}(\phi): \mathcal{F}(V) \rightarrow \mathcal{F}(U), \ \ [X \rightarrow V] \longmapsto [U \times_V X \rightarrow U]
\end{equation} where $ U \times_V X $ is the fibered product given by pulling back the family $ X \rightarrow V $ along the morphism $\phi: U \rightarrow V$:
\begin{equation}
\begin{tikzpicture}
\matrix (m) [matrix of math nodes,row sep=2em,column sep=3em,minimum width=1em] {
	U \times_V X  & X \\
	U & V\\};
\path[-stealth]
(m-1-1) edge  node [left] {{\small $ $}} (m-2-1)
(m-1-1.east|-m-1-2) edge  node [above] {{\small $  $}} (m-1-2)
(m-2-1.east|-m-2-2) edge  node [below] {} node [below] {{\small $ f $}} (m-2-2)
(m-1-2) edge  node [right] {{\small $  $}} (m-2-2);
%edge [dashed,-] (m-2-1);
\end{tikzpicture}
\end{equation}
\end{example}
With the above formalism in hand, the existence of a fine moduli space, therefore, corresponds to the \textit{representability} of the moduli functor $\mathcal{F}$ in the sense that 
\begin{equation}
\mathcal{F} = Mor_{Sch_{\mathbb{C}}} (\cdot, \mathcal{M}) \ \ for \ some \ \mathcal{M}\in Sch_{\mathbb{C}}.
\end{equation}If this is the case, then we say that \textit{$\mathcal{F}$ is represented by $\mathcal{M}$}.
\newpage
\begin{remark} Let  $\mathcal{F}: \mathcal{C}^{op} \longrightarrow Sets$ be a moduli functor represented by an object $M$, then one can recast the desired properties of being a "fine" moduli space %indicated in the wish-list above 
	as follows:
	\begin{enumerate}
		\item Take $B:=spec(\mathcal{\mathbb{C}})= \{ \ast \}$, then from the representability we have \begin{equation}
		\mathcal{F}(\{ \ast \}) \cong h_M(\{ \ast \})=Mor_{\mathcal{C}}(\{ \ast \}, M). 
		\end{equation} Note that the RHS is just the set of (closed) points of $ M$, and LHS is the set of corresponding isomorphism classes.
		\item When $B:=M$, then we get an isomorphism 
		\begin{equation}
		\mathcal{F}(M) \cong h_M(M)=Mor_{\mathcal{C}}(M, M), 
		\end{equation} which allows us to define the universal object $\mathcal{T}$ to be the object corresponding to the identity morphism $id_M \in Mor_{\mathcal{C}}(M, M).$ 
	\end{enumerate}
	
\end{remark}
\noindent These observations yield the following corollary:
\begin{corollary}
	If $\mathcal{F}: \mathcal{C}^{op} \longrightarrow  Sets$ is a moduli functor represented by an object $\mathcal{M}$ in $\mathcal{C}$, then there exists an one-to-one correspondence between the set $\{X\rightarrow B \}$ of (equivalences classes of) families  and $Mor_{\mathcal{C}}(B,\mathcal{M})$. That is, 	\begin{equation}
	\{X\rightarrow B\}/\sim \ \longleftrightarrow Mor_{\mathcal{C}}(B,\mathcal{M}). 
	\end{equation} Furthermore, for a morphism $f:B \rightarrow \mathcal{M}$ corresponding to the equivalence class $ \big[X \rightarrow B\big]$ of the family $ X \rightarrow B, $ we have 
	\begin{equation}
	\big[X \rightarrow B\big] = \big[B \times_{\mathcal{M}} \mathcal{T} \rightarrow B \big].
	\end{equation}
\end{corollary}

In many cases, however, a moduli functor is \textit{not} representable in the category $Sch$ of schemes. This is the place where the notion of a \textit{stack} comes into play. In that situation, one can still make sense of the notion of a  moduli space in a weaker sense. This version, namely \textit{a coarse moduli space}, is still efficient enough to encode the isomorphism classes of points. That is, it has the correct points, and captures the geometry of moduli space. However, the sensitivity on the behavior of arbitrary families is no longer available. In other words, a coarse moduli space may not be able to distinguish two non-isomorphic families in many cases. Hence, the classification in this "family-wise" level is by no means possible. To elaborate the last statement, we first introduce the formal definition of a so-called \emph{coarse moduli space}, and then we shall provide two important examples: $(i)$ \textit{the moduli problem of classifying vector bundles of  fixed rank over an algebraic curve over a field $k$} \cite{Neumann}, and $(ii)$ \textit{the moduli of elliptic curves} \cite{BenZ,Hain1}.

\begin{definition}  Let $\mathcal{C}:=Sch_{\mathbb{C}}$ for the sake of simplicity.
	\textit{A coarse moduli space} for a moduli functor $\mathcal{F}: \mathcal{C}^{op} \longrightarrow  Sets$ consists of a pair $ (M,\psi) $ where   $M$ is an object in $\mathcal{C}$ , and  $\psi: \mathcal{F}\rightarrow h_M$ is a natural transformation such that \begin{enumerate}
		\item $\psi_{spec(\mathbb{C})}: \mathcal{F}(spec(\mathbb{C})) \rightarrow h_M(spec(\mathbb{C}))$ is a bijection of sets.
		\item Such a pair $ (M,\psi) $ satisfies the following \textit{universal property}: For any scheme $N$ and any natural transformation $\phi: \mathcal{F}\rightarrow h_N$, there exists a unique morphism $f:M \rightarrow N$ such that the following diagram commutes:
		\begin{equation}
		\begin{tikzpicture}
		\matrix (m) [matrix of math nodes,row sep=2em,column sep=3em,minimum width=1em] {
			\mathcal{F}  & h_M \\
			& h_N \\};
		\path[-stealth]
		(m-1-1) edge  node [left] {{\small $ \phi $}} (m-2-2)
		edge  node [above] {{\small $ \psi $}} (m-1-2)
		%(m-2-1.east|-m-2-2) edge  node [below] {} node [below] {{\small $ f $}} (m-2-2)
		(m-1-2) edge [dashed,->] node  [right] {{\small $\exists h_f $}} (m-2-2);
		%edge [dashed,-] (m-2-1);
		\end{tikzpicture}
		\end{equation}
	\end{enumerate}
\end{definition}
\begin{remark} Here, $h_f: h_M \rightarrow h_N$ is the associated natural transformation of functors \ref{defn_Yonede functor h}.
	Second condition also implies that if it exists, a coarse moduli space $M$ for a moduli functor $\mathcal{F}$ is unique up to a unique isomorphism.
\end{remark}
\newpage
\begin{proposition}
	\cite{Hoskins} Let $ (M,\psi) $ be a coarse moduli space for a moduli functor $\mathcal{F}: \mathcal{C}^{op} \longrightarrow  Sets$ where $M$ is a scheme and $\psi: \mathcal{F}\rightarrow h_M$ is the corresponding natural transformation. Then $ (M,\psi) $ is a fine moduli space if and only if the following conditions hold: \begin{enumerate}
		\item There exists a family $\mathcal{T} \rightarrow M$ such that $\psi_M (\mathcal{T})=id_M \in Mor_{\mathcal{C}}(M,M)$. 
		\item For families $X\rightarrow B$ and $Y\rightarrow B$ on a base scheme $B$, \begin{equation}
		[X\rightarrow B]=[Y\rightarrow B] \Longleftrightarrow \psi_B(X)= \psi_B (Y).
		\end{equation}
	\end{enumerate}
\end{proposition}
\begin{proof}
	It follows directly from the definition of a fine moduli space.
\end{proof}

\subsection{Moduli of vector bundles of fixed rank}
We would like to investigate the moduli problem of classifying vector bundles of fixed rank over a smooth, projective algebraic curve $X$ of genus $ g $ over a field $ k $ with char $k=0$.   We define the corresponding moduli functor 
\begin{equation}
\mathcal{F}_X^n:  Sch_{\mathbb{C}}^{op} \longrightarrow Sets
\end{equation} as follows: To each object $U$ in $Sch_{\mathbb{C}}$, $ \mathcal{F}_X^n $ assigns   the set $ \mathcal{F}_X^n(U) $ of isomorphism classes of families of vector bundles of rank $n$ on $X$ parametrized by $U$. That is,
\begin{equation*}
\mathcal{F}_X^n(U)= \big\{ E\rightarrow U \times_{spec \mathbb{C}} X : E \text{ is a vector bundle of rank n }\big\}/\sim
\end{equation*} Here, we say that two vector bundles $\pi: E \rightarrow U \times_{spec \mathbb{C}} X$ and $\pi': E' \rightarrow U \times_{spec \mathbb{C}} X$ over $U \times_{spec \mathbb{C}} X$ are \textit{equivalent} if and only if there exists an isomorphism $f: E \xrightarrow{\sim} E'$ of vector bundles such that the following diagram commutes:
\begin{equation}
\begin{tikzpicture}
\matrix (m) [matrix of math nodes,row sep=2em,column sep=3em,minimum width=1em] {
	E  & & E' \\
	&  U \times_{spec\mathbb{C}} X &\\};
\path[-stealth]
(m-1-1) edge  node [left] {{\small $ \pi $}} (m-2-2)
(m-1-1.east|-m-1-3) edge  node [above] {{\small $ f $}} (m-1-3)
%(m-2-1.east|-m-2-2) edge  node [below] {} node [below] {{\small $ f $}} (m-2-2)
(m-1-3) edge  node [right] {{\small $ \pi' $}} (m-2-2);
%edge [dashed,-] (m-2-1);
\end{tikzpicture}
\end{equation}
To each morphism $f:U\rightarrow V$ in $Sch_{\mathbb{C}}$,  \ $ \mathcal{F}_X^n $ assigns the map of vector bundles  \begin{equation}
\mathcal{F}_X^n(f): \mathcal{F}_X^n(V)\rightarrow\mathcal{F}_X^n(U)
\end{equation} which is defined by pulling back of the vector bundle $E\rightarrow V \times_{spec\mathbb{C}} X$ along the morphism $f \times id_X$. Note that $U \times_{spec \mathbb{C}} X$ is just the usual direct product $U \times X$ with the projection maps $pr_1: U \times X \rightarrow U$ and $pr_2: U \times X \rightarrow X$ such that \begin{equation}
\begin{tikzpicture}
\matrix (m) [matrix of math nodes,row sep=3em,column sep=4em,minimum width=2em] {
	U \times_{spec \mathbb{C}} X  & X \\
	U & spec \mathbb{C}\\};
\path[-stealth]
(m-1-1) edge  node [left] {{\small $ pr_1 $}} (m-2-1)
(m-1-1.east|-m-1-2) edge  node [above] {{\small $ pr_2 $}} (m-1-2)
(m-2-1.east|-m-2-2) edge  node [below] {} node [below] {{\small $  $}} (m-2-2)
(m-1-2) edge  node [right] {{\small $  $}} (m-2-2);
%edge [dashed,-] (m-2-1);
\end{tikzpicture}
\end{equation} Now, we would like to show that $\mathcal{F}_X^n$ can not be representable by some scheme $M.$
\vspace{5pt}

\noindent \textbf{Claim:} \textit{$ \mathcal{F}_X^n $ is not representable in $Sch_{\mathbb{C}}$.}
\begin{proof}
	Assume that $ \mathcal{F}_X^n $ is representable by a scheme $M$. That is, we have a natural isomorphism \begin{equation}
	\mathcal{F}_X^n \cong h_M.
	\end{equation} Let $U$ be a scheme and $ E\in\mathcal{F}_X^n (U)$. Then, we have a vector bundle $E\rightarrow U \times X$ of rank $n$. Let $L$ be a line bundle over $U$, then we can define \textit{the induced bundle $ pr_1^* L $ on $U \times X$} by pulling back $L$ along the projection map $pr_1$. Therefore, we obtain a particular vector bundle \begin{equation}
	E':= E \otimes pr_1^* L \longrightarrow U \times X.
	\end{equation} Indeed, $E'$ is \textit{a twisted bundle} where each fiber of $E'$ is obtained by multiplying each fiber of $E$ by a scalar. Hence, we produce a new vector bundle $E'$   by "twisting" $E$ such that $E'$ is \textit{not} (globally) isomorphic  to $E$. Let $\{U_i\}$ be a local trivializing cover of $U$ such that $L|_{U_i}$ is trivial. Then it follows from the definition of $E'$ that \begin{equation}
	E' |_{U_i \times X} \cong E |_{U_i \times X} \ \ \forall i.
	\end{equation} As  $ \mathcal{F}_X^n $ is representable by a scheme $M$, there are morphisms $f_E: U \rightarrow M$ and $f_{E'}: U \rightarrow M$ in $Mor_{Sch_{\mathbb{C}}}(U,M)$ corresponding to $E$ and $E'$ respectively such that \begin{equation}
	f_E |_{U_i \times X} = f_{E'} |_{U_i \times X} \ \ \forall i.
	\end{equation} Since $h_M$ is a sheaf, it follows from the gluing axiom that all such morphisms are glued together nicely such that  \begin{equation}
	f_E |_{U \times X} = f_{E'} |_{U \times X}.
	\end{equation} But, from the representability of the functor $ \mathcal{F}_X^n $, it implies that  \begin{equation}
	E' |_{U \times X} \cong E |_{U \times X}.
	\end{equation} This yields a contradiction to $E\ncong E'$.
\end{proof}

\begin{remark}
	The reason behind the failure of the representability of $ \mathcal{F}_X^n $ is that vector bundles have a number of non-trivial automorphisms, for instance, induced by a scalar multiplication as above. This example, in fact, may provide an important insight why a generic moduli problem is destined to be non-representable in the category of schemes. In many cases, the main source of non-representability problem turns out to be \textit{the existance of non-trivial automorphisms for the moduli problem of interest. }
\end{remark}

\subsection{Moduli of elliptic curves}

In this section, we study the moduli space of elliptic curves, and try to show how the existence of non-trivial automorphisms again leads to the non-representability of the corresponding moduli functor. The study of such a moduli problem is indeed a classical topic, and further discussion can be found elsewhere. For an introduction, we refer to \cite{BenZ}. More detailed treatment (possibly with different approaches) can be found, for instance, in \cite{Harris and Morrison, Hain1, Hain2}.

\begin{definition}
	We first recall that one can define the notion of \textit{an elliptic curve over $\mathbb{C}$} in a number of equivalent ways. \textbf{An elliptic curve over $\mathbb{C}$} is defined to be either of the following objects:\begin{enumerate}
		\item A Riemannian surface $\Sigma$ of genus 1 with a choice of a point $p \in \Sigma.$
		\item A quotient space $\mathbb{C}/\Lambda$ where $\Lambda= \omega_1 \cdot \mathbb{Z} \oplus \omega_2 \cdot \mathbb{Z}$ is a rank 2 lattice in $\mathbb{C}$ for  each $\omega_i \in \mathbb{C}$.
		\item A smooth algebraic curve of genus 1 and degree 3 in $\mathbb{P}^2_{\mathbb{C}}$.
	\end{enumerate} 
	
\end{definition}

We actually use the second characterization of an elliptic curve, namely the one given in terms of lattices. With this interpretation in hand, the study of the moduli of elliptic curves boils down to the study of integer lattices of full rank  in $\mathbb{C}$.

\paragraph{$SL_2(\mathbb{Z})$-action on the upper half plane and the fundamental domain.} Recall that the Lie group  $SL_2(\mathbb{Z})$ acts on the upper half plane $\mathbb{H}=\{z\in \mathbb{C} : im(z)>0 \}$ as follows:  \begin{equation*}
{\small \begin{bmatrix}
	a & b \\ c & d
	\end{bmatrix}} \cdot z:= \dfrac{az+b}{cz+d}
\ \ \ \text{for all } {\small \begin{bmatrix}
	a & b \\ c & d
	\end{bmatrix}}  \in SL_2(\mathbb{Z}) \text{ and } \forall z \in \mathbb{H}. \end{equation*}  It is clear from the definition of the action that both $\pm I$ act on $\mathbb{H}$ in the same way, and hence we will concentrate on the action of $PSL_2(\mathbb{Z}):= SL_2(\mathbb{Z})/\{\pm I\}$ on $\mathbb{H}$. Then \textit{the fundamental domain $\Gamma := \mathbb{H}/PSL_2(\mathbb{Z})$} of this action turns out to be the set \begin{equation*}
\Gamma= \Big\{z : -\frac{1}{2}\leq Re(z) < \frac{1}{2}, \ |z|>1\Big\} \cup \Big\{z : -\frac{1}{2}\leq Re(z) \leq 0, \ \ |z|=1  \Big\}. 
\end{equation*}

\begin{figure}[h]
	
	\centering
	\includegraphics[width=0.4\linewidth]{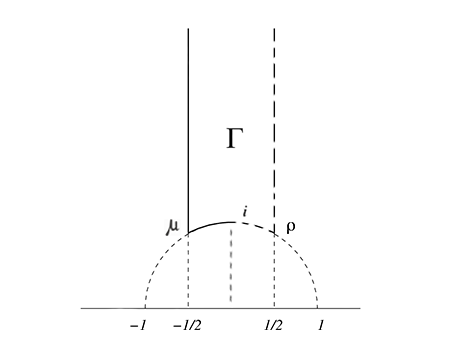}
	\caption{The fundamental domain $\Gamma$  where $\mu=e^{2\pi i/3}$ and $\rho=e^{2\pi i/6}$.}
	\label{fig:fundamental-domain1}
\end{figure}
It is very well known that $\Gamma$ is in fact a Riemann
surface whose points correspond to the isomorphism classes of lattices of full rank in $\mathbb{C}$ up to  homotheties. Note also that any lattice $\Lambda=\omega_1\mathbb{Z} \oplus \omega_2\mathbb{Z}$ is isomorphic to a "normalized" lattice\begin{equation}
\Lambda_{\tau}:= 1 \cdot \mathbb{Z} \oplus \tau \cdot \mathbb{Z} \ \text{ for some } \tau \in \mathbb{H}.
\end{equation}We say that two lattices $\Lambda_{\tau_1}= 1 \cdot \mathbb{Z} \oplus \tau_1 \cdot \mathbb{Z}$ and $\Lambda_{\tau_2}= 1 \cdot \mathbb{Z} \oplus \tau_2 \cdot \mathbb{Z}$ with $\tau_i \in \mathbb{H}$ are  \textit{homothetic} if there exits $g \in PSL_2(\mathbb{Z})$ such that $\tau_2 = g \cdot \tau_1.$ In other words, $\mathbb{H}/PSL_2(\mathbb{Z})$ serves as \textit{a coarse moduli space} for isomorphism classes of elliptic curves $\mathbb{C}/\Lambda_{\tau}$ with $\tau \in \Gamma$. As we shall see soon, however, it turns out that the space $\Gamma$ is not sensitive enough to parametrize certain families of elliptic curves. This amounts to say that not all families of elliptic curves over some base $B$ correspond to morphisms $B\rightarrow \mathbb{H}/PSL_2(\mathbb{Z})$. Hence, $\Gamma$ fails to become a fine moduli space.

\begin{remark}
	\begin{enumerate}
		\item []
		\item The fundamental domain $\Gamma$ can also be represented as a free product of finite groups $\mathbb{Z}_2$ and $\mathbb{Z}_3$ as follows: 
		\begin{align*}
		\Gamma= & \Big \langle S:= {\small \begin{bmatrix}
			0 & -1 \\ 1 & 0
			\end{bmatrix}}, T:={\small \begin{bmatrix}
			1 & 1 \\ 1 & 0
			\end{bmatrix}}\in PSL_2(\mathbb{Z}) \ | \ S^2=(ST)^3=I \Big \rangle \nonumber \\  \cong & \ \mathbb{Z}_2 \star \mathbb{Z}_3.
		\end{align*}
		\item The action of $PSL_2(\mathbb{Z})$ on $\mathbb{H}$ is not free. Indeed, some routine computations show that \begin{equation} 
		Stab_{PSL_2(\mathbb{Z})} (\tau) \cong \begin{cases}
		\mathbb{Z}_2, &\ \tau=i  \\
		\mathbb{Z}_3, &\ \tau= \mu \ or \ \rho \\
		\{e\}, & \ else.
		\end{cases}
		\end{equation}
		\item It follows from the non-freeness of the action that one has special types of lattices, namely \textit{the square lattice} $\Lambda_i$ and the \textit{hexagonal lattice} $\Lambda_{\mu}$ (or $\Lambda_{\rho}$) such that \begin{equation} 
		Aut(\Lambda_{\tau})=\begin{cases}
		\mathbb{Z}_4, &\ \tau=i  \\
		\mathbb{Z}_6, &\ \tau= \mu \ or \ \rho.
		\end{cases}
		\end{equation} This gives rise to non-trivial automorphisms for the corresponding elliptic curves $\mathbb{C}/\Lambda_i$ and $\mathbb{C}/\Lambda_{\mu}$ by using, for instance, rotational symmetries of a square and that of a hexagon respectively.   As before,  \textit{the existance of  non-trivial automorphisms} allows us to produce some examples  which eventually  show that the corresponding moduli problem of elliptic curves can not be represented by the space  $\mathbb{H}/PSL_2(\mathbb{Z})$.
		But, we should keep in mind that it becomes \textit{the coarse moduli space.}
	\end{enumerate}
	
\end{remark}
\newpage
\paragraph{Moduli functor for the families of elliptic curves.} We define the corresponding moduli functor  
\begin{equation}
\mathcal{F}_{ell}:  Sch_{\mathbb{C}}^{op} \longrightarrow Sets, \ U \mapsto \mathcal{F}_{ell}(U)
\end{equation} as follows: Given a scheme $U$, one can define $\mathcal{F}_{ell}(U):=S_{ell}(U)/\sim$ where $S_{ell}(U)$ is the set of (continuous) families of elliptic curves over the base scheme $U$: \begin{equation}
S_{ell}(U):= \bigg\{ E \rightarrow U:  \text{each fiber $ E_u $ is $ \mathbb{C}/\Lambda_{\tau(u)} \ \forall u\in U$} \bigg\}
\end{equation} where $ \mathbb{C}/\Lambda_{\tau(u)} $ is an elliptic curve with $\tau: U \rightarrow \mathbb{H}/PSL_2(\mathbb{Z}), \ u \mapsto \tau(u)$, and $E=\bigsqcup_{u \in U}E_u$. We say that two families $\pi_E: E \rightarrow U$ and $\pi_{E'}: E' \rightarrow U$ over $U$ are \textit{equivalent} if and only if there exists an isomorphism $f: E \xrightarrow{\sim} E'$ of families such that on each fiber $E_u=\pi_E^{-1}(u)$, $f$ restricts to an automorphism of elliptic curves 
\begin{equation}
f|_{E_u}: E_u \xrightarrow{\sim} E'_u, 
\end{equation}and the following diagram commutes:
\begin{equation}
\begin{tikzpicture}
\matrix (m) [matrix of math nodes,row sep=2em,column sep=3em,minimum width=1em] {
	E  & & E' \\
	& U &\\};
\path[-stealth]
(m-1-1) edge  node [left] {{\small $ \pi_E $}} (m-2-2)
(m-1-1.east|-m-1-3) edge  node [above] {{\small $ f $}} (m-1-3)
%(m-2-1.east|-m-2-2) edge  node [below] {} node [below] {{\small $ f $}} (m-2-2)
(m-1-3) edge  node [right] {{\small $ \pi_{E'} $}} (m-2-2);
%edge [dashed,-] (m-2-1);
\end{tikzpicture}
\end{equation} Note that from the previous discussion, it is not hard to observe that the space $\mathbb{H}/PSL_2(\mathbb{Z})$ in fact serves as the desired coarse moduli space for the moduli functor above. We now would like to show that the moduli functor $\mathcal{F}_{ell}$ is not representable by $\mathbb{H}/PSL_2(\mathbb{Z})$.
\vspace{5pt}

\noindent \textbf{Claim:} The moduli functor $\mathcal{F}_{ell}$ is not representable by $\mathbb{H}/PSL_2(\mathbb{Z})$.

\begin{proof}
	Assume the contrary. Then, we have the following  one-to-one correspondence between two sets:
	\begin{equation} \label{correspondence}
	Mor_{Sch_{\mathbb{C}}}(U, \mathbb{H}/PSL_2(\mathbb{Z})) \cong \mathcal{F}_{ell}(U).
	%\Big \{U \rightarrow \mathbb{H}/PSL_2(\mathbb{Z})\Big\} \longleftrightarrow \Big\{\text{isomorphism classes of families of elliptic curves over U} \Big\}
	\end{equation} 
	We first consider a "constant" family $E$ of elliptic curves on the interval $[0,1]$ where each fiber $E_x$ is of the form \begin{equation}
	E_x:= \mathbb{C}/\Lambda_{i} \ \text{ for all } x.
	\end{equation}	Recall that $Aut(\mathbb{C}/\Lambda_{i})\cong \mathbb{Z}_4$. Let $f$ be a non-trivial automorphism of $\mathbb{C}/\Lambda_{i}$ given as a multiplication by $i$, 
	\begin{equation}
	f: \mathbb{C}/\Lambda_{i} \rightarrow \mathbb{C}/\Lambda_{i}, \ z \mapsto iz.
	\end{equation}	Then one can identify the fibers $E_0$ and $E_1$ along the morphism $f$ so that a particular family $\mathcal{E}$ of elliptic curves over $\mathbb{S}^1$ can be obtained. Similarly, one can construct another family $\mathcal{E}'$ of elliptic curves over $\mathbb{S}^1$ by gluing the fibers $E_0$ and $E_1$ via the identity morphism. We then obtain two \textit{non-isomorphic families} $\mathcal{E}$ and $\mathcal{E}'$ of elliptic curves over $\mathbb{S}^1$ with the generic fibers being all isomorphic. That is, $ [\mathcal{E}] \neq [\mathcal{E}'] $ such that $$\mathcal{E}_x\cong \mathcal{E}_x' \cong E_x \ \text{ for all } x\in (0,1),$$ where $\mathcal{E}_x$ and $\mathcal{E}_x'$ denote the fibers of "twisted" and "trivial" families respectively. See Figure \ref{fig:families}.
	\begin{figure}[h]
		\centering
		\includegraphics[width=0.55\linewidth]{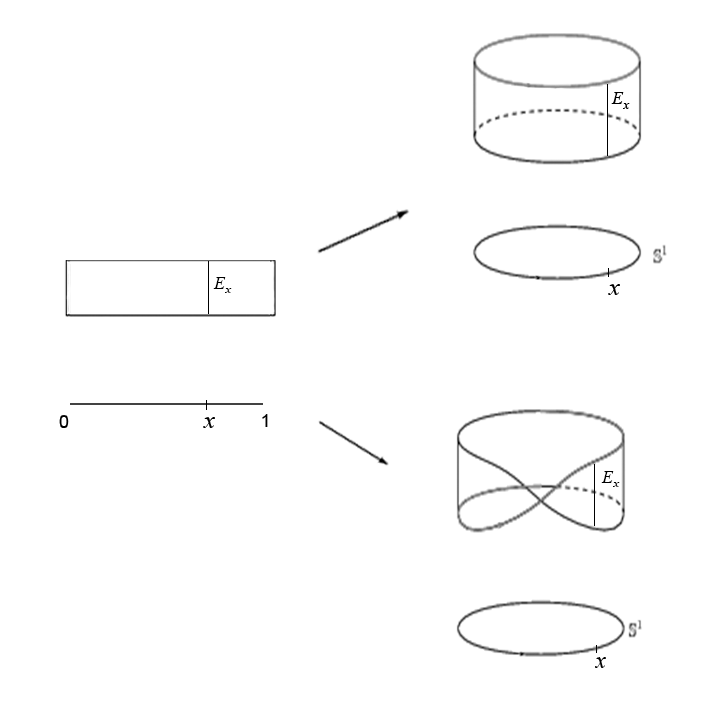}
		\caption{The "trivial" and "twisted" families of elliptic curves over $\mathbb{S}^1$ with generically isomorphic fibers.}
		\label{fig:families}
	\end{figure}
	\vspace{5pt}
	
	As $\mathcal{F}_{ell}$ is representable by $\mathbb{H}/PSL_2(\mathbb{Z})$, there are corresponding morphisms $f_{\mathcal{E}}, \ f_{\mathcal{E}'}: \mathbb{S}^1 \rightarrow \Gamma$ for $\mathcal{E}$ and $\mathcal{E}'$ respectively such that \begin{equation}
	f_{\mathcal{E}}, \ f_{\mathcal{E}'}: \mathbb{S}^1 \longrightarrow \mathbb{H}/PSL_2(\mathbb{Z}), \ s \mapsto [i].
	\end{equation}That is, each representing morphism is just the constant map. Let $\{U_k\}$ be a local trivializing cover for $\mathbb{S}^1$ such that \begin{equation}
	\pi_{\mathcal{E}}^{-1} (U_k) \cong (\mathbb{C}/\Lambda_i)  \times U_k \cong \pi_{\mathcal{E}'}^{-1} (U_k).
	\end{equation} Then the representability condition implies that the representing morphisms are locally the same as well. That is,  \begin{equation}
	f_{\mathcal{E}} |_{U_k} = f_{\mathcal{E}'} |_{U_k} \ \ \forall k.
	\end{equation} Since $h_{\mathbb{H}/PSL_2(\mathbb{Z})}=Mor_{Sch_{\mathbb{C}}}( - ,\mathbb{H}/PSL_2(\mathbb{Z}))$ is a sheaf, it follows from the gluing axiom that all such morphisms are glued together nicely. That is,  \begin{equation}
	f_{\mathcal{E}} = f_{\mathcal{E}'} \ \ on \ \ \mathbb{S}^1.
	\end{equation} But, it follows from the representability of the functor $ \mathcal{F}_{ell} $ that we must have \begin{equation}
	[\mathcal{E}] = [\mathcal{E}'],
	\end{equation} which is a contradiction.
	
\end{proof}

\begin{remark} \begin{enumerate} \item []
		\item The construction above shows that the correspondence \ref{correspondence} is not good enough to distinguish the "trivial" and "twisted" families of isomorphism classes of elliptic curves over $\mathbb{S}^1$ with generically isomorphic fibers. As before, the main source of this failure is due to the existence of non-trivial automorphism group for the fibers of the form  $ \mathbb{C}/\Lambda_{i}.$ %Recall that $Aut(\mathbb{C}/\Lambda_{i})\cong \mathbb{Z}_4$. Indeed, 
		The existence of non-trivial automorphisms, on the other hand, is due to the fact that $PSL_2(\mathbb{Z})$ acts on $\mathbb{H}$ non-freely.
		\item One way of circumventing this sort of problem is to change the way of organizing the moduli data. For instance, we can use the language of stacks, and redefine the moduli problem as a certain gruopoid-valued "functor" \begin{equation}
		\mathcal{F}: \mathcal{C}^{op} \longrightarrow \ Grpds
		\end{equation} 
		where \textbf{$Grpds$} denotes \textit{the 2-category of groupoids} with objects being categories $\mathcal{C}$ in which all morphisms are isomorphisms (these sorts of categories are called \textit{groupoids}), 1-morphisms being functors $\mathcal{F}: \mathcal{C} \rightarrow \mathcal{D}$ between groupoids, and 2-morphisms being natural transformations $\psi: \mathcal{F} \Rightarrow \mathcal{F'}$ between two functors. Note that the groupoid structure, in fact, allows us to keep track of the non-trivial automorphisms as a part of the moduli data.
		\item When we go back the example above, one can easily check that the space $\mathbb{H}/PSL_2(\mathbb{Z})$ can, in fact, be regarded as a \textit{groupoid} 	where  objects are the elements of $\mathbb{H}$, and the set of morphisms is the set $PSL_2(\mathbb{Z}) \times \mathbb{H}.$ In fact,
		\begin{align}
		Mor_{\mathbb{H}/PSL_2(\mathbb{Z})}(x,y):= & \big\{ g \in PSL_2(\mathbb{Z}):  y=g\cdot x  \big \} \nonumber\\
		\cong & \ PSL_2(\mathbb{Z}) \times \mathbb{H}.
		\end{align}
		Denote a morphism by $(g,x)$: \ $x\mapsto g\cdot x.$ Note that two morphisms $(g,x)$ and $(h,y)$ are composable if $x=h\cdot y$. Then we have \begin{equation}
		(g,h\cdot y) \circ (h,y) = (gh, y).
		\end{equation} Furthermore, the inverse of the morphism $(g,x)$ is $(g^{-1},g\cdot x)$. Informally speaking, two non-isomorphic families $\mathcal{E}$ and $\mathcal{E}'$ above can be represented by points like $[i, f]$ and $[i,id]$ via suitable constant representing morphisms as above where the second slots in the parenthesis are to keep track of possible automorphisms distinguishing these families. The last statement will be elaborated in the next section. In literature, the space $\mathbb{H}/PSL_2(\mathbb{Z})$ is  an example of an \textit{orbifold}.
		
	\end{enumerate}
	
\end{remark}

\section{2-categories and Stacks} Stacks and 2-categories serve as motivating/prototype conceptual examples before introducing the notions like \textit{$ \infty $-categories, derived schemes, higher stacks and derived stacks} \cite{Vezz2}. By using 2-categorical version of the Yoneda lemma \cite{Neumann}, one can show that the "refined" moduli functor \begin{equation}
\mathcal{F}: \mathcal{C}^{op} \longrightarrow \ Grpds
\end{equation}  turns out to be representable in the 2-category $Stks$ of stacks.  The price we have to pay is to adopt higher categorical dictionary, which leads to the change in the level of abstraction in a way that objects under consideration may become rather counter-intuitive. We first briefly recall the basics of 2-category theory. For details, we refer to \cite{Stacks, Neumann}.
\subsection{A digression on 2-categories}
\begin{definition}
	\textit{	A 2-category $\mathcal{C}$} consists of the following data:
	\begin{enumerate}
		\item A collection of objects: $Ob(\mathcal{C})$.
		\item For each pair $x,y$ of objects, \textit{a category} $Mor_{\mathcal{C}}(x,y)$. Here, objects of the category $Hom_{\mathcal{C}}(x,y)$ are called \textit{1-morphisms} and are denoted either by $f:x \rightarrow y$ or $ x \xrightarrow{f} y $. The morphisms of $Mor_{\mathcal{C}}(x,y)$ are called \textit{2-morphisms} and are denoted either by $\phi : f \Rightarrow g$ or \begin{equation}
		{\small \begin{tikzpicture}
			\matrix[matrix of nodes,column sep=1.7cm] (cd)
			{
				$ x $ & $ y. $ \\
			};
			\draw[->] (cd-1-1) to[bend left=50] node[label=above:$ {f} $] (U) {} (cd-1-2);
			\draw[->] (cd-1-1) to[bend right=50,name=D] node[label=below:${g}$] (V) {} (cd-1-2);
			\draw[double,double equal sign distance,-implies,shorten >=5pt,shorten <=5pt] 
			(U) -- node[label=left:${\tiny \phi}$] {} (V);
			\end{tikzpicture}}
		\end{equation}The composition of two 2-morphisms $\alpha : f \Rightarrow g$ and $\beta : g \Rightarrow h$ in $Mor_{\mathcal{C}}(x,y)$ is called \textit{a vertical composition}, and denoted by $\beta \circ \alpha: f \Rightarrow h$ or
		\begin{equation}
		{\small \begin{tikzpicture}
			\matrix[matrix of nodes,column sep=2.4cm] (cd)
			{
				$ x $ & $ y, $ \\
			};
			\draw[->] (cd-1-1) to[bend left=55] node[label=above:$ {f} $] (U) {} (cd-1-2);
			\draw[->] (cd-1-1) to[bend right=55,name=D] node[label=below:${h}$] (V) {} (cd-1-2);
			\draw[double,double equal sign distance,-implies,shorten >=22pt,shorten <=-1pt] 
			(U) -- node[label=below left:${\tiny \beta}$] {} (V);
			\draw[->] (cd-1-1) to [bend right=3] node[label= below right:${g}$] (V) {} (cd-1-2);
			\draw[double,double equal sign distance,-implies,shorten >=-22pt,shorten <=25pt] 
			(U) -- node[label=left:${\tiny \alpha }$] {} (V);
			\end{tikzpicture}}
		{\small \begin{tikzpicture}
			\matrix[matrix of nodes,column sep=2.4cm] (cd)
			{
				$ x $ & $ y. $ \\
			};
			\draw[->] (cd-1-1) to[bend left=55] node[label=above:$ {f} $] (U) {} (cd-1-2);
			\draw[->] (cd-1-1) to[bend right=55,name=D] node[label=below:${h}$] (V) {} (cd-1-2);
			%	\draw[double,double equal sign distance,-implies,shorten >=22pt,shorten <=-1pt] 
			%(U) -- node[label=below left:${\tiny \beta}$] {} (V);
			%\draw[->] (cd-1-1) to [bend right=3] node[label= below right:${g}$] (V) {} (cd-1-2);
			\draw[double,double equal sign distance,-implies,shorten >=2pt,shorten <=5pt] 
			(U) -- node[label=left:${\tiny \beta \circ \alpha }$] {} (V);
			\end{tikzpicture}}
		\end{equation} A 2-morphism $\alpha:f \Rightarrow g$ is \textit{invertible} if there exits a 2-morphism $\beta: g \Rightarrow f$ such that $\beta \circ \alpha = id_f$ and $\alpha \circ \beta = id_g$. Furthermore, an invertible 2-morphism $\alpha: f \Rightarrow g$ is called a \textit{2-isomorphism}. It is sometimes denoted by $\alpha: f \Leftrightarrow g$.
		\item For each triple $x,y,z$ of objects in $\mathcal{C}$, there is \textit{a composition functor} 
		\begin{equation}
		\mu_{x,y,z}: Mor_{\mathcal{C}}(x,y) \times Mor_{\mathcal{C}}(y,z) \rightarrow Mor_{\mathcal{C}}(x,z)
		\end{equation} which is defined as follows: \begin{enumerate}
			\item On 1-morphisms, it acts as the usual composition of morphisms in $\mathcal{C}$: $$\mu_{x,y,z}: \big(x \xrightarrow{f} y, y \xrightarrow{g}z \big) \longmapsto \big(y \xrightarrow{g \circ f} z\big)$$
			\item On 2-morphisms, it acts as a \textit{horizontal composition}, denoted by $\star$: \begin{equation}
			\mu_{x,y,z}: \big(f \xRightarrow{\alpha} f', g \xRightarrow{\alpha'}g' \big) \longmapsto \big(g \circ f \xRightarrow{\alpha' \star \alpha} g' \circ f'\big)
			\end{equation}
			That is, given two 2-morphisms \begin{equation}
			{\small \begin{tikzpicture}
				\matrix[matrix of nodes,column sep=1.7cm] (cd)
				{
					$ x $ & $ y,$ \\
				};
				\draw[->] (cd-1-1) to[bend left=50] node[label=above:$ {f} $] (U) {} (cd-1-2);
				\draw[->] (cd-1-1) to[bend right=50,name=D] node[label=below:${f'}$] (V) {} (cd-1-2);
				\draw[double,double equal sign distance,-implies,shorten >=5pt,shorten <=5pt] 
				(U) -- node[label=left:${\tiny \alpha}$] {} (V);
				\end{tikzpicture}} {\small \begin{tikzpicture}
				\matrix[matrix of nodes,column sep=1.7cm] (cd)
				{
					$  \ y $ & $ z, $ \\
				};
				\draw[->] (cd-1-1) to[bend left=50] node[label=above:$ {g} $] (U) {} (cd-1-2);
				\draw[->] (cd-1-1) to[bend right=50,name=D] node[label=below:${g'}$] (V) {} (cd-1-2);
				\draw[double,double equal sign distance,-implies,shorten >=5pt,shorten <=5pt] 
				(U) -- node[label=left:${\tiny \alpha'}$] {} (V);
				\end{tikzpicture}}
			\end{equation}
			
			$ \mu_{x,y,z} $ maps the pair $ \big(f \xRightarrow{\alpha} f', g \xRightarrow{\alpha'}g' \big) $ of 2-morphisms to a 2-morphism 
			\begin{equation}
			{\small \begin{tikzpicture}
				\matrix[matrix of nodes,column sep=3.1cm] (cd)
				{
					$ x $ & $ z. $ \\
				};
				\draw[->] (cd-1-1) to[bend left=50] node[label=above:$ {g\circ f} $] (U) {} (cd-1-2);
				\draw[->] (cd-1-1) to[bend right=50,name=D] node[label=below:${g'\circ f'}$] (V) {} (cd-1-2);
				\draw[double,double equal sign distance,-implies,shorten >=5pt,shorten <=5pt] 
				(U) -- node[label=left:${\tiny \alpha' \star \alpha}$] {} (V);
				\end{tikzpicture}}
			\end{equation}
		\end{enumerate}
	\end{enumerate} These data must satisfy the following conditions:
	\begin{itemize}
		\item For each object $X$ of $\mathcal{C}$ and each 1-morphism $f:A \rightarrow B$, we have an identity 1-morphism $id_X: X \rightarrow X$ and an identity 2-morphism $id_f: f \Rightarrow f$ respectively.
		\item The composition of 1-morphisms (2-morphisms respectively) is associative.
		\item Horizontal and vertical compositions of 2-morphisms are "compatible" in the following sense. For a composition diagram
		\begin{equation}
		{\small \begin{tikzpicture}
			\matrix[matrix of nodes,column sep=2.2cm] (cd)
			{
				$ x $ & $ y $ & $ z $ \\
			};
			\draw[->] (cd-1-1) to[bend left=55] node[label=above right:$ {f} $] (U) {} (cd-1-2);
			
			\draw[->] (cd-1-1) to[bend right=65,name=D] node[label=below right:${f''}$] (V) {} (cd-1-2);
			\draw[double,double equal sign distance,-implies,shorten >=22pt,shorten <=-1pt] 
			(U) -- node[label=below left:${\tiny \alpha'}$] {} (V);
			\draw[->] (cd-1-1) to [bend right=3] node[label=  right:${\tiny f'}$] (V) {} (cd-1-2);
			\draw[double,double equal sign distance,-implies,shorten >=-23pt,shorten <=25pt] 
			(U) -- node[label=left:${\tiny \alpha }$] {} (V);
			
			\draw[->] (cd-1-2) to[bend left=55] node[label=above:$ {g} $] (U) {} (cd-1-3);
			\draw[->] (cd-1-2) to[bend right=55,name=D] node[label=below:${g''}$] (V) {} (cd-1-3);
			\draw[double,double equal sign distance,-implies,shorten >=22pt,shorten <=-1pt] 
			(U) -- node[label=below left:${\tiny \beta'}$] {} (V);
			\draw[->] (cd-1-2) to [bend right=3] node[label= right:${g'}$] (V) {} (cd-1-3);
			\draw[double,double equal sign distance,-implies,shorten >=-22pt,shorten <=25pt] 
			(U) -- node[label=left:${\tiny \beta }$] {} (V);
			\end{tikzpicture}}
		\end{equation}we have 
		\begin{equation}
		(\beta' \circ \beta) \star (\alpha' \circ \alpha) = (\beta' \star \alpha') \circ (\beta \star \alpha).
		\end{equation}
		
	\end{itemize}
\end{definition}

\begin{example}
	Every category can be realized as a 2-category. Indeed, let $\mathcal{C}$ be a category and $Mor_{\mathcal{C}}(A,B)$ denote the \textit{set} of morphisms between two objects $A,B$. Then it is clear to observe that for any pair $(A,B)$ of objects in $\mathcal{C}$ the set $Mor_{\mathcal{C}}(A,B)$ defines a category \textbf{Mor}$_{\mathcal{C}}$(A,B) whose objects (1-morphisms) are just morphisms $A \rightarrow B$ in $\mathcal{C}$, and morphisms (2-morphisms) are just identities. That is, there are no non-trivial higher morphisms in this realization. A category is sometimes called a \textit{1-category}.
\end{example}

\begin{remark}
	Given a 2-category $\mathcal{C}$, one can obtain a category $C_0$ by defining $Ob(C_0):=Ob(\mathcal{C})$ and the "set" $Mor_{C_0}(A,B)$ of morphisms in $C_0$ to be \begin{equation}
	Mor_{C_0} (A,B):= Mor_{\mathcal{C}}(A,B)/ \sim
	\end{equation} where $f \sim g$ if there exits a 2-isomorphism $\alpha : f \Leftrightarrow g$. That is, $Mor_{C_0}(A,B)$ is just the set of isomorphism classes of 1-morphisms in $\mathcal{C}$.
\end{remark}

\begin{example}
	A collection of categories forms a 2-category, namely the 2-category $Cat$ of categories. Here, \textit{objects} of $Cat$ are just categories $\mathcal{C}$, \textit{1-morphisms} in $Cat$ are functors $\mathcal{F}: \mathcal{C} \rightarrow \mathcal{D}$ between two categories, and \textit{2-morphisms} are natural transformations $\eta: \mathcal{F} \Rightarrow \mathcal{G}$ of functors. In this example, there are no non-trivial higher $n$-morphisms for $n>2$.  Once we allow such types of morphisms, we land in the territory of higher categories.
\end{example}

\begin{definition}
	Let $\mathcal{C}$ be a 2-category. Two objects $X,Y$ in $\mathcal{C}$ are said to be \textit{equivalent} if there exist a pair $(X \xrightarrow{f} Y, Y \xrightarrow{g} X)$ of 1-morphisms, and two 2-isomorphisms $\alpha: g \circ f \Leftrightarrow id_X$ and $\alpha': f \circ g \Leftrightarrow id_Y$. 
\end{definition}

\begin{definition}
	\textit{A pseudo-functor $\mathcal{F}: \mathcal{C} \rightarrow \mathcal{D}$ between two 2-categories} $\mathcal{C}, \mathcal{D}$ consists of the following data:
	\begin{enumerate}
		\item For each object $A$ in $\mathcal{C}$, \ an object $ \mathcal{F}(A) $ in $ \mathcal{D}$,
		\item For each 1-morphism $A \xrightarrow{f} B$ in $\mathcal{C}$, \ a 1-morphism $\mathcal{F}(A) \xrightarrow{\mathcal{F}(f)} \mathcal{F}(B)$ in $\mathcal{D}$,
		\item For each 2-morphism $\alpha: f \Rightarrow g$ in $\mathcal{C}$, \ a 2-morphism $\mathcal{F}(\alpha): \mathcal{F}(f) \Rightarrow \mathcal{F}(g)$ in $\mathcal{D}$ satisfying the following conditions: \begin{enumerate}
			\item $\mathcal{F}$ respects 1- and 2-identities: $\mathcal{F}(id_A)=id_{\mathcal{F}(A)}$ and $\mathcal{F}(id_f)=id_{\mathcal{F}(f)}$ for all $A\in Ob(\mathcal{C})$ and $f \in Mor_{\mathcal{C}}(X,Y)$.
			\item $\mathcal{F}$ respects a composition of 1-morphisms up to a 2-isomorphism: Given a composition of 1-morphisms $ A \xrightarrow{f} B \xrightarrow{g} C$ in $\mathcal{C}$, there is a 2-isomorphism
			\begin{equation}
			\phi_{g,f}^{\mathcal{F}}:  \mathcal{F}(g \circ f) \Rightarrow \mathcal{F}(g) \circ \mathcal{F}(f)
			\end{equation} such that the following diagram commutes \textit{(encoding the associativity)}:
			
			\begin{equation} 
			\begin{tikzpicture}
			\matrix (m) [matrix of math nodes,row sep=2em,column sep=3em,minimum width=1em] {
				\mathcal{F}(h \circ g \circ f)  & \mathcal{F}(h) \circ \mathcal{F}(g \circ f)  \\
				\mathcal{F}(h \circ g) \circ \mathcal{F}(f)& \mathcal{F}(h)\circ \mathcal{F}(g)\circ \mathcal{F}(f)\\};
			\path[-stealth]
			(m-1-1) edge [double] node [left] {{\small $ \phi_{h\circ g,f}^{\mathcal{F}} $}} (m-2-1)
			edge [double] node [below] {{\small $ \phi_{h, g \circ f}^{\mathcal{F}} $}} (m-1-2)
			(m-2-1.east|-m-2-2) edge [double] node [below] {} node [below] {{\footnotesize $ \phi_{h,g}^{\mathcal{F}} \star id_{\mathcal{F}(f)}$}} (m-2-2)
			(m-1-2) edge [double] node [right] {{\small $ id_{\mathcal{F}(h)} \star \phi_{g,f}^{\mathcal{F}}$}} (m-2-2);
			%edge [dashed,-] (m-2-1);
			\end{tikzpicture}
			\end{equation} such that $\phi_{f,id_A}^{\mathcal{F}}= \phi_{id_B,f}^{\mathcal{F}}= id_{\mathcal{F}(f)}$.
			\item $\mathcal{F}$ respects both vertical and horizontal compositions: Given a vertical composition $\beta \circ \alpha: f \Rightarrow h$ with
			\begin{equation}
			{\small \begin{tikzpicture}
				\matrix[matrix of nodes,column sep=2.2cm] (cd)
				{
					$ x $ & $ y, $ \\
				};
				\draw[->] (cd-1-1) to[bend left=55] node[label=above:$ {f} $] (U) {} (cd-1-2);
				\draw[->] (cd-1-1) to[bend right=55,name=D] node[label=below:${h}$] (V) {} (cd-1-2);
				\draw[double,double equal sign distance,-implies,shorten >=22pt,shorten <=-1pt] 
				(U) -- node[label=below left:${\tiny \beta}$] {} (V);
				\draw[->] (cd-1-1) to [bend right=3] node[label=  right:${g}$] (V) {} (cd-1-2);
				\draw[double,double equal sign distance,-implies,shorten >=-22pt,shorten <=25pt] 
				(U) -- node[label=left:${\tiny \alpha }$] {} (V);
				\end{tikzpicture}}
			\end{equation} we have $\mathcal{F}(\beta \circ \alpha)= \mathcal{F}(\beta) \circ \mathcal{F}(\alpha)$. Given a horizontal composition \begin{equation}
			{\small \begin{tikzpicture}
				\matrix[matrix of nodes,column sep=2.8cm] (cd)
				{
					$ x $ & $ z, $ \\
				};
				\draw[->] (cd-1-1) to[bend left=50] node[label=above:$ {g\circ f} $] (U) {} (cd-1-2);
				\draw[->] (cd-1-1) to[bend right=50,name=D] node[label=below:${g'\circ f'}$] (V) {} (cd-1-2);
				\draw[double,double equal sign distance,-implies,shorten >=5pt,shorten <=5pt] 
				(U) -- node[label=left:${\tiny \alpha' \star \alpha}$] {} (V);
				\end{tikzpicture}}
			\end{equation} with $\alpha: f \Rightarrow f'$ and $\alpha': g \Rightarrow g'$, we have the following commutative diagram:
			\begin{equation} 
			\begin{tikzpicture}
			\matrix (m) [matrix of math nodes,row sep=3em,column sep=4em,minimum width=2em] {
				\mathcal{F}(g)\circ	\mathcal{F}(f)  & \mathcal{F}(g')\circ	\mathcal{F}(f')  \\
				\mathcal{F}(g \circ f) & \mathcal{F}(g' \circ f')\\};
			\path[-stealth]
			(m-1-1) edge [double] node [left] {{\small $ \phi_{g,f}^{\mathcal{F}} $}} (m-2-1)
			edge [double] node [below] {{\small $ \mathcal{F}(\alpha') \star \mathcal{F}(\alpha) $}} (m-1-2)
			(m-2-1.east|-m-2-2) edge [double] node [below] {} node [below] {{\footnotesize $ \mathcal{F}(\alpha'\star \alpha) $}} (m-2-2)
			(m-1-2) edge [double] node [right] {{\small $\phi_{g',f'}^{\mathcal{F}}$}} (m-2-2);
			%edge [dashed,-] (m-2-1);
			\end{tikzpicture}
			\end{equation}
			
		\end{enumerate}
	\end{enumerate}
\end{definition}

\begin{definition}
\begin{enumerate} \item []
	\item 	A \textbf{prestack} on $\mathcal{C}$ is defined to be  a particular (contravariant) pseudo-functor $\mathcal{F}: \mathcal{C} \rightarrow \mathcal{D}$ where $\mathcal{C}$ is an ordinary category and $\mathcal{D}$ is the 2-category $Grpds$ of groupoids. That is, it is a pseudo-functor $\mathcal{F}: \mathcal{C}^{op} \rightarrow Grpds$ for some category $\mathcal{C}$. Recall that, in the 2-category $Grpds$ of groupoids, objects are categories $\mathcal{C}$ in which all morphisms are isomorphisms (these sorts of categories are called \textit{groupoids}), 1-morphisms are functors $\mathcal{F}: \mathcal{C} \rightarrow \mathcal{D}$ between groupoids, and 2-morphisms are  natural transformations $\psi: \mathcal{F} \Rightarrow \mathcal{F'}$ between two functors. \item Let $\mathcal{C}_{\tau}$ be a site. That is, $\mathcal{C}_{\tau}$ is a category $\mathcal{C}$ endowed with a Grothendieck topology $\tau$. Then, \textbf{a stack} on $\mathcal{C}_{\tau}$ is just a prestack with local-to-global properties w.r.t. $\tau$. So, it can be considered as "a sheaf of groupoids" in a suitable sense.
\end{enumerate}
\end{definition}

\subsection{2-category of Stacks and 2-Yoneda's Lemma} We like to present how stacks over a site $\mathcal{C}_{\tau}$ form a 2-category. To this end, we need to introduce the notions of 1- and 2-morphisms between two stacks.

\begin{definition}
	Let $\mathcal{C}$ be a category, and $\mathcal{X}, \mathcal{Y}: \mathcal{C}^{op} \longrightarrow Grpds $ be two prestacks. \textit{A 1-morphism $F: \mathcal{X} \rightarrow \mathcal{Y}$ of prestacks} consists of the following data:
	\begin{enumerate}
		\item For each object $A$ in $\mathcal{C}$, \ \textit{a functor} $F_A: \mathcal{X}(A) \rightarrow \mathcal{Y}(A)$,
		\item For each morphism $f: A \rightarrow B$ in $\mathcal{C}$, \ a 2-isomorphism \begin{equation}
		{\small \begin{tikzpicture}
			\matrix[matrix of nodes,column sep=1.7cm] (cd)
			{
				$ \mathcal{X}(B) $ & $ \mathcal{Y}(A) $ \\
			};
			\draw[->] (cd-1-1) to[bend left=50] node[label=above:$ {F_A \circ \mathcal{X}(f)} $] (U) {} (cd-1-2);
			\draw[->] (cd-1-1) to[bend right=50,name=D] node[label=below:${\mathcal{Y}(f) \circ F_B}$] (V) {} (cd-1-2);
			\draw[double,double equal sign distance,-implies,shorten >=5pt,shorten <=5pt] 
			(U) -- node[label=left:${\tiny F_f}$] {} (V);
			\end{tikzpicture}}
		\end{equation} such that the following  diagram commutes  up to a 2-isomorphism $$F_f: F_A \circ \mathcal{X}(f) \Longrightarrow \mathcal{Y}(f) \circ F_B$$ \ in $Grpds$: 
		\begin{equation} 
		\begin{tikzpicture}
		\matrix (m) [matrix of math nodes,row sep=3em,column sep=4em,minimum width=2em] {
			\mathcal{X}(B)  & \mathcal{Y}(B)  \\
			\mathcal{X}(A) & \mathcal{Y}(A)\\};
		\path[-stealth]
		(m-1-1) edge  node [left] {{\small $ \mathcal{X}(f) $}} (m-2-1)
		edge  node [below] {{\small $ F_B $}} (m-1-2)
		(m-2-1.east|-m-2-2) edge  node [below] {} node [below] {{\footnotesize $ F_A $}} (m-2-2)
		(m-1-2) edge  node [right] {{\small $\mathcal{Y}(f)$}} (m-2-2);
		%edge [dashed,-] (m-2-1);
		\end{tikzpicture}
		\end{equation}
		
		\item Given a composition of 1-morphisms $ A \xrightarrow{f} B \xrightarrow{g} C$ in $\mathcal{C}$, the corresponding 2-isomorphisms $F_f$ and $F_g$ define $F_{g \circ f}$ in compatible with the natural 2-isomorphisms $\phi_{g,f}^{\mathcal{X}}$ and $\phi_{g,f}^{\mathcal{Y}}$. Indeed, using horizontal and vertical compositions of 2-morphisms, $F_{g \circ f}: F_A \circ \mathcal{X}(g \circ f) \Rightarrow \mathcal{Y}(g \circ f) \circ F_C$ is given by \begin{equation*}
		F_{g \circ f}= \big(\phi_{g,f}^{\mathcal{Y}} \star id_{F_C}\big) \circ \big(id_{\mathcal{Y}(f)} \star F_g \big) \circ \big(F_f \star id_{\mathcal{X}(g)}\big) \circ \big(id_{F_A} \star \phi_{g,f}^{\mathcal{X}}\big).
		\end{equation*}
	\end{enumerate}Such a 1-morphism is just \textit{a natural transformation} between two pseudo-functors of 2-categories. 
\end{definition}

\begin{definition}
	Let $\mathcal{C}$ be a category, and $F, G: \mathcal{X} \rightarrow \mathcal{Y}$ be a pair of 1-morphisms of prestacks. \textit{A 2-morphism of between $F$ and $G$} 
	\begin{equation}
	{\small \begin{tikzpicture}
		\matrix[matrix of nodes,column sep=1.7cm] (cd)
		{
			$ \mathcal{X} $ & $ \mathcal{Y} $ \\
		};
		\draw[->] (cd-1-1) to[bend left=50] node[label=above:$ {F} $] (U) {} (cd-1-2);
		\draw[->] (cd-1-1) to[bend right=50,name=D] node[label=below:${G}$] (V) {} (cd-1-2);
		\draw[double,double equal sign distance,-implies,shorten >=5pt,shorten <=5pt] 
		(U) -- node[label=left:${\tiny \Psi}$] {} (V);
		\end{tikzpicture}}
	\end{equation} is a collection $\big\{\Psi_A: F_A \Rightarrow G_A : A\in Ob(\mathcal{C}) \big\}$ of invertible natural transformations of functors.
\end{definition}

\begin{remark}
	A 1-morphism (2-morphism respectively) of stacks is defined as a 1-morphism (2-morphism resp.) of the underlying prestacks. 
\end{remark}

\begin{proposition} \cite{Neumann}
	\begin{enumerate}
		\item Given a category $\mathcal{C}$, prestacks over $\mathcal{C}$ form a 2-category $PreStk_{\mathcal{C}}$ of prestacks where 1- and 2-morphisms are defined as above.
		\item Stacks over a site $\mathcal{C}$ form a 2-category $Stk_{\mathcal{C}}$ of stacks with 1- and 2-morphisms being as above.
	\end{enumerate}
\end{proposition}
\noindent Furthermore, it follows from the construction that we have the following observations \cite{Neumann}.

\begin{proposition} \label{sheaves as full 2-succategory of Stks}
	Let $\mathcal{C}$ be a category. Then the (1-) category $PreShv_{\mathcal{C}}$ of presheaves (of sets) over $\mathcal{C}$ is a full 2-subcategory of $PreStk_{\mathcal{C}}$. In addition, if $\mathcal{C}$ admits a site structure, then  the (1-) category $Shv_{\mathcal{C}}$ of sheaves (of sets) over $\mathcal{C}$ is a full 2-subcategory of $Stk_{\mathcal{C}}$.
\end{proposition}

\begin{remark}
	As we have already discussed in Remark \ref{rmks on Yoneda's lemma}, Yoneda's lemma \ref{Yoneda's lemma} implies that any object $X$ in a category $\mathcal{C}$ can be understood by studying all morphism into it. In other words, for any $X\in Ob(\mathcal{C})$ one has a particular \textit{sheaf of sets}	\begin{equation}
	h_X: \mathcal{C}^{op} \longrightarrow \ Sets, \ \ \ Y \longmapsto h_X(Y):=Mor_{\mathcal{C}}(Y,X)
	\end{equation} which uniquely determines $X$. If $\mathcal{C}$ admits a suitable site structure, then it follows from Proposition \ref{sheaves as full 2-succategory of Stks} that the sheaf $Mor_{\mathcal{C}}(\cdot,X)$ can be considered as a stack with trivial 2-morphisms. We denote this stack by $\underline{X}$.
	
\end{remark}

\begin{lemma} (2-categorical Yoneda's Lemma for prestacks)
	Let $\mathcal{Y}: \mathcal{C}^{op} \longrightarrow Grpds$ be a prestack over a category $\mathcal{C}$. Then for each object $X$ in $\mathcal{C}$, there exits an equivalence of categories \begin{equation}
	\mathcal{Y} (X) \cong Mor_{PreStk_{\mathcal{C}}} ( \underline{X}, \mathcal{Y}).
	\end{equation}
\end{lemma}
\begin{proof}
	First, let us try to understand the objects of interests in the statement. On the left hand side, $ \mathcal{Y} (X) $ is a groupoid, i.e. a category for which all morphisms are isomorphisms. On the right hand side, $ Mor_{PreStk_{\mathcal{C}}} ( \underline{X}, \mathcal{Y}) $ is the category of morphisms in the 2-category $PreStk_{\mathcal{C}}$. 1-morphisms are a collection $$\big\{F_A: \underline{X}(A) \rightarrow \mathcal{Y}(A) : A\in Ob(\mathcal{C}) \big\}$$ of functors with some compatibility conditions as above. Such collection is denoted by $F: \underline{X} \rightarrow \mathcal{Y}$. On the other hand, 2-morphisms are of the form \begin{equation}
	{\small \begin{tikzpicture}
		\matrix[matrix of nodes,column sep=1.7cm] (cd)
		{
			$ \underline{X} $ & $ \mathcal{Y} $ \\
		};
		\draw[->] (cd-1-1) to[bend left=50] node[label=above:$ {F} $] (U) {} (cd-1-2);
		\draw[->] (cd-1-1) to[bend right=50,name=D] node[label=below:${G}$] (V) {} (cd-1-2);
		\draw[double,double equal sign distance,-implies,shorten >=5pt,shorten <=5pt] 
		(U) -- node[label=left:${\tiny \Psi}$] {} (V);
		\end{tikzpicture}}
	\end{equation} where for each object $Z$, $\Phi_Z: F_Z \Rightarrow G_Z$ is an invertible natural transformation. Therefore, for any $f \in \underline{X}(Z)$, we have an isomorphism $\Psi_{Z, f}: F_Z(f) \xrightarrow{\sim} G_Z (f)$ in $\mathcal{Y}(Z)$. To show the desired equivalence, we introduce the following functors:\begin{enumerate}
		\item We define the functor $\Theta: Mor_{PreStk_{\mathcal{C}}} ( \underline{X}, \mathcal{Y}) \longrightarrow \mathcal{Y} (X)$ as follows:
		\begin{enumerate}
			\item On objects (1-morphisms), \ $\big(F: \underline{X} \rightarrow \mathcal{Y} \big) \longmapsto F_X(id_X)$
			\item On morphisms (2-morphisms), \ $$\big(\Psi: F \Rightarrow G\big) \longmapsto \big(\Psi_{X, id_X}: F_X(id_X) \xrightarrow{\sim} G_X(id_X)\big)$$ 
		\end{enumerate}
		\item Define the functor $\eta: \mathcal{Y} (X) \longrightarrow Mor_{PreStk_{\mathcal{C}}} ( \underline{X}, \mathcal{Y})$ as follows: \begin{enumerate}
			\item On objects $A$ in $\mathcal{Y}(X)$, \ $ A \longmapsto \big(F^{(A)}: \underline{X} \rightarrow \mathcal{Y} \big)$. \ Here $F^{(A)}$ is given as a collection $\{F_U^{(A)}: U\in Ob(\mathcal{C}) \}$ of functors such that \begin{equation}
			F_U^{(A)}: \underline{X}(U) \rightarrow \mathcal{Y}(U), \ (U \xrightarrow{f} X) \longmapsto \mathcal{Y}(f)(A) 
			\end{equation} where $\mathcal{Y}(f): \mathcal{Y}(X) \rightarrow \mathcal{Y}(U)$. Notice that $ \underline{X}(U)=Mor_{\mathcal{C}}(U,X) $ is just a \textit{set} in the first place, but, as we remarked before, it can be viewed as a category for which all morphisms are identities. Therefore, $F_U^{(A)}$ acts on morphisms of $\underline{X}(U)$ trivially. That is, it maps $id_f$ to $id_{F^{(A)}_U(f)}$. 
			\item $\eta$ sends morphisms $\varphi: A \xrightarrow{\sim} B$ in $\mathcal{Y}(X)$ to 2-morphisms \begin{equation}
			{\small \begin{tikzpicture}
				\matrix[matrix of nodes,column sep=2.4cm] (cd)
				{
					$ \underline{X} $ & $ \mathcal{Y} $ \\
				};
				\draw[->] (cd-1-1) to[bend left=50] node[label=above:$ {F^{(A)}} $] (U) {} (cd-1-2);
				\draw[->] (cd-1-1) to[bend right=50,name=D] node[label=below:${F^{(B)}}$] (V) {} (cd-1-2);
				\draw[double,double equal sign distance,-implies,shorten >=5pt,shorten <=5pt] 
				(U) -- node[label=left:${\tiny \Psi^{(\varphi)}}$] {} (V);
				\end{tikzpicture}}
			\end{equation} where $ \Psi^{(\varphi)} $ is a collection $\big\{\Psi^{(\varphi)}_U: F_U^{(A)} \Longrightarrow F_U^{(B)} : U\in Ob(\mathcal{C}) \big\}$ of invertible natural transformations of functors where  for each object  $f: U \rightarrow X$ in $\underline{X}(U)$, there is an isomorphism \begin{equation*}
			\Psi^{(\varphi)}_U(f): F_U ^{(A)} (f) \longrightarrow F_U^{(B)} (f), \ \Psi^{(\varphi)}_U(f):= \mathcal{Y}(f)(\varphi).
			\end{equation*} Notice that $\mathcal{Y}(f): \mathcal{Y}(X) \rightarrow \mathcal{Y}(U)$ is a functor of groupoids, and hence it maps  $\varphi$ to an isomorphism.
		\end{enumerate}
		\item Now, we like to show that the compositions of $\eta$ and $\Theta$ are identities up to 2-isomorphisms. \begin{enumerate}
			\item Let $A$ be an object in $\mathcal{Y}(X)$. Then we have \begin{align*}
			\big(\Theta \circ \eta \big) (A) &= \Theta \big(F^{(A)}: \underline{X} \rightarrow \mathcal{Y} \big) \\ & = F_X^{(A)} (id_X) \\ & = \mathcal{Y}(id_X)(A) \ \text{ by definition of } F_X^{(A)} \\ & = id_{\mathcal{Y}(X)} (A) = A \ \text{ since $ \mathcal{Y}(id_X) $ is a functor of groupoids}  
			\end{align*}
			\item Let $F: \underline{X} \rightarrow \mathcal{Y}$ be an object in $Mor_{PreStk_{\mathcal{C}}}( \underline{X}, \mathcal{Y})$. Then we get \begin{equation}
			\big(\eta \circ \Theta \big)\big(F\big) = \eta (F_X (id_X)) = F^{(F_X (id_X))}
			\end{equation} where the 1-morphism $F^{(F_X (id_X))}$ is defined by, for all $U \in Ob(\mathcal{C})$, \begin{equation*}
			F_U^{(F_X (id_X))}: \underline{X}(U) \rightarrow \mathcal{Y}(U), \ (U \xrightarrow{f} X) \longmapsto \mathcal{Y}(f)({F_X (id_X)}) 
			\end{equation*} \textit{Claim.} $\mathcal{Y}(f)({F_X (id_X)}) \cong F_U (f)$.
			\begin{proof}
				It follows directly from the definition of a 2-morphism that for a morphism $ U \xrightarrow{f} X $, there exists $$F_f: F_U \circ \underline{X}(f) \Longrightarrow \mathcal{Y}(f) \circ F_X$$ such that the following  diagram commutes  up to a 2-isomorphism  $F_f$:
				\begin{equation} 
				\begin{tikzpicture}
				\matrix (m) [matrix of math nodes,row sep=3em,column sep=4em,minimum width=2em] {
					\underline{X}(X)  & \mathcal{Y}(X)  \\
					\underline{X}(U) & \mathcal{Y}(U)\\};
				\path[-stealth]
				(m-1-1) edge  node [left] {{\small $ \underline{X}(f) $}} (m-2-1)
				edge  node [below] {{\small $ F_X $}} (m-1-2)
				(m-2-1.east|-m-2-2) edge  node [below] {} node [below] {{\footnotesize $ F_U $}} (m-2-2)
				(m-1-2) edge  node [right] {{\small $\mathcal{Y}(f)$}} (m-2-2);
				%edge [dashed,-] (m-2-1);
				\end{tikzpicture}
				\end{equation}where $\underline{X}(f): \underline{X}(X) \rightarrow \underline{X}(U), \ g\mapsto g\circ f$. Therefore, we have \begin{equation}
				\big(\mathcal{Y}(f) \circ F_X\big) ( id_X) \cong F_U (id_X \circ f)=F_U ( f),
				\end{equation} which proves the claim.
			\end{proof} Therefore, since the claim holds for all $U$, we conclude that \begin{equation}
			\big(\eta \circ \Theta \big)\big(F\big) = F
			\end{equation} As a result, we get the desired equivalence of categories \begin{equation}
			\eta: \mathcal{Y} (X) \xleftrightarrow{\sim} Mor_{PreStk_{\mathcal{C}}} ( \underline{X}, \mathcal{Y}) : \Theta.
			\end{equation}
		\end{enumerate}
	\end{enumerate}
\end{proof}

\begin{remark}
	2-categorical Yoneda's lemma implies that if $\mathcal{Y}: \mathcal{C}^{op} \longrightarrow Grpds$ is a moduli functor for some moduli problem, then \textit{it is always representable by $\mathcal{Y}$ in the 2-category of (pre-)stacks. }
\end{remark}

%-----------------------------------------------------
%\newpage
%\appendix

\bibliography{xbib}

\end{document}